\documentclass{amsart}
\calclayout
 \usepackage{amssymb,amsmath,amsfonts,epsfig,latexsym,tikz}
 \usepackage{tikz-cd}
 \usepackage{tikz}
\usepackage{enumerate}
\usepackage{mathtools}
\usepackage[shortlabels]{enumitem}
\usepackage{subcaption}
\usepackage{tabto}

\usetikzlibrary{positioning}
\usetikzlibrary{matrix}
\usetikzlibrary{decorations}
\usetikzlibrary{decorations.pathreplacing, decorations.pathmorphing, angles,quotes}
 
 \newtheorem{theorem}{Theorem}[section]
\newtheorem{definition}[theorem]{Definition}
\newtheorem{proposition}[theorem]{Proposition}
\newtheorem{lemma}[theorem]{Lemma}
\newtheorem{corollary}[theorem]{Corollary}

\theoremstyle{definition}
\newtheorem{remark}[theorem]{Remark}
\newtheorem{example}[theorem]{Example}

\def\R{\mathbb{R}}

\def\1{\mathbf{1}}

\def\<{\langle}
\def\>{\rangle}

\DeclareMathOperator{\lk}{lk}

\DeclareMathOperator{\Star}{Star}
\DeclareMathOperator{\supp}{supp}

\newcommand\nullset\varnothing

\makeatletter
\def\maketag@@@#1{\hbox{\m@th\normalfont\normalsize#1}}
\makeatother

\begin{document}

\title[Resolutions of local face modules]{Resolutions of local face modules, functoriality, and vanishing of local $h$-vectors}           
    
\author{Matt Larson, Sam Payne, and Alan Stapledon}

\address{Stanford U. Department of Mathematics, 450 Jane Stanford Way, Stanford, CA 94305}
\email{mwlarson@stanford.edu}

\address{UT Department of Mathematics, 2515 Speedway, RLM 8.100, Austin, TX 78712}
\email{sampayne@utexas.edu}

\address{Sydney Mathematics Research Institute, L4.42, Quadrangle A14, University of Sydney, NSW 2006, Australia}
\email{astapldn@gmail.com}

\begin{abstract}
We study the local face modules of triangulations of simplices, i.e., the modules over face rings whose Hilbert functions are local $h$-vectors. In particular, we give resolutions of these modules by subcomplexes of Koszul complexes as well as  functorial maps between modules induced by inclusions of faces.  As applications, we prove a new monotonicity result for local $h$-vectors and new results on the structure of faces in triangulations with vanishing local $h$-vectors.
\end{abstract}
 
\maketitle

\vspace{-20 pt}

\section{Introduction}

In this paper, we study the modules over face rings, introduced by Athanasiadis and Stanley, whose Hilbert functions are the relative local $h$-vectors of quasi-geometric homology triangulations of simplices, a broad class of formal subdivisions that includes all geometric triangulations and is natural from the point of view of combinatorial commutative algebra.  See Section 2.1 for the precise definition and further references.

Fix an infinite field $k$.  Let $\sigma \colon \Gamma \to 2^V$ be a quasi-geometric homology triangulation of a simplex, and let $E$ be a face of $\Gamma$. Say that a face $G \in \Gamma$ is \emph{interior} if $\sigma(G) = V$, and let $I$ be the ideal in the face ring $k[\lk_\Gamma(E)]$ generated by the faces that are interior relative to $E$, i.e., 
\[
 I = (x^F : F \sqcup E \mbox{ is interior} ).
\]
Let $d = |V|-|E|$, which is the Krull dimension of $k[\lk_\Gamma(E)]$, and let $\theta_1, \ldots, \theta_{d}$ be a special l.s.o.p., as in \cite{Stanley92, Athanasiadis12b}. See also \S\ref{ss:sr}, where we recall the definition and construction of special l.s.o.p.s.

\begin{definition}
The \emph{local face module} $L(\Gamma,E)$ is the image of $I$ in $k[\lk_\Gamma(E)]/(\theta_1, \ldots, \theta_d)$.
\end{definition}

\noindent Note that $L(\Gamma,E)$ is a finite dimensional graded $k$-vector space. The \emph{local $h$-vector} is its Hilbert function: 
\[
\ell(\Gamma, E) := (\ell_0, \ldots, \ell_{d}), \quad \quad \mbox{ where } \ \ell_i := \dim L(\Gamma,E)_i.
\]
The local face module $L(\Gamma,E)$ depends on the choice of a special l.s.o.p., but $\ell(\Gamma,E)$ is an invariant of the triangulation with the symmetry $\ell_i = \ell_{d- i}$. See \S\ref{sec:triangulations} for details and references. In the past few years, there has been significant research activity on the combinatorics of local $h$-vectors and relations to intersection homology \cite{Athanasiadis16, KatzStapledon16, Stapledon17, deCataldoMiglioriniMustata18}.  Recent advances include a proof that every non-negative integer vector satisfying $\ell_0 = 0$ and $\ell_i = \ell_{d-i}$ is the local $h$-vector of a quasi-geometric triangulation for $E = \emptyset$ \cite{JKMS}, and a relative hard Lefschetz theorem that yields unimodality of local $h$-vectors for regular subdivisions in a more general setting (for regular nonsimplicial polyhedral subdivisions that are not necessarily rational) \cite{Karu19}.

\medskip

Here, we investigate the local face modules $L(\Gamma, E)$ using methods of combinatorial commutative algebra. In particular, we describe natural combinatorial resolutions of these modules as well as natural maps of $k[\lk_\Gamma(E)]$-modules, $L(\Gamma,E) \to L(\Gamma, E')$, for $E \subset E'$.  Our first theorem gives explicit generators for the kernel of the natural map $I \to k[\lk_\Gamma(E)]/(\theta_1, \ldots, \theta_d)$. Moreover, we extend this to an exact sequence of graded $k[\lk_\Gamma(E)]$-modules in which each term is a direct sum of degree-shifted monomial ideals.

Label the vertices of the simplex $V = \{ v_1, \ldots, v_n \}$.  For a subset $U \subset V$, let $U^c := V \smallsetminus U$. After relabeling, we may assume that $\sigma(E)^c = \{v_1, \ldots, v_b\}$. Given $S \subset \{v_1, \ldots, v_d\}$, 
we define the ideal $I_S \subset k[\lk_\Gamma(E)]$ by
\begin{equation*}
I_S := ( x^F : \, \sigma(F \sqcup E)^c \subset S). 
\end{equation*}
Note that $I_{S'} \subset I_{S}$ for $S' \subset S$, and $I_S$ depends only on $S \cap \{v_1, \ldots, v_b\}$. For instance, $I_{\emptyset} = I$ and $I_{S} = k[\lk_{\Gamma}(E)]$ if $\{v_1, \ldots, v_b \} \subset S$. By the definition of a special l.s.o.p. (Definition~\ref{def:special}), after reordering, we may assume
\[
\supp(\theta_i) \subset \{ w \in \lk_\Gamma(E) : v_i \in \sigma(w) \}, 
\]
for $1 \leq i \leq b$.  As a consequence, for any $v_i \in S$, multiplication by $\theta_i$ induces a degree 1 map $\lambda_i \colon I_S \to I_{S \smallsetminus \{v_i\}}$.

\begin{theorem}\label{thm:resolution}
There is an exact sequence of graded $k[\lk_{\Gamma}(E)]$-modules
 $$0  \to k[\lk_{\Gamma}(E)][-d] \to \bigoplus_{ |S|  = d - 1 } I_{S}[-(d-1)] \to \dotsb  \to \bigoplus_{|S| = 1} I_{S}[-1] \to I \to L(\Gamma,E) \to 0,$$
where,  
for $S= \{v_{i_0}, \ldots, v_{i_k}\}$, with $i_0 < \cdots < i_k$, the differential restricted to $I_S$ is $\oplus_{j = 0}^k (-1)^j \lambda_{i_j}$.
\end{theorem}

\begin{corollary}\label{cor:presentation}
The kernel of the surjection $I \to L(\Gamma, E)$ is the ideal $J$ generated by 
 $$ \left \{ \theta_i \cdot x^{F} : F \sqcup E \mbox{ is interior } \right \}  \cup  \left \{ \theta_{j} \cdot x^G : \sigma(G \sqcup E) = \{v_j\}^c, \mbox{ for } 1 \leq j \leq b \right \}.$$
\end{corollary}

We also construct maps between local face modules, as follows. For faces $E \subset E'$ in $\Gamma$, let $\Star(E' \smallsetminus E)$ denote the closed star of $E' \smallsetminus E$ in $\lk_\Gamma(E)$. 
We have a natural inclusion of complexes 
$\lk_\Gamma(E') \subset \lk_\Gamma(E)$. 



\begin{theorem}\label{thm:maps}
Let $E \subset E'$ be faces of $\Gamma$, with $$d = n - \vert E \vert, \quad  
d' = n - \vert E' \vert, \quad 
\mbox{ and } \quad b' = n - |\sigma(E')|.$$ Let $\{\theta_1, \dotsc, \theta_{d } \}$ be a special l.s.o.p. for $k[\lk_{\Gamma}(E)]$,  
and let 
$\theta_i' := \theta_i|_{\Star(E' \smallsetminus E)}$.  Then there is a 
unique homomorphism of graded 
$k$-algebras 
\[
\phi\colon k[\lk_{\Gamma}(E)]/(\theta_1, \dotsc, \theta_d) \to k[\lk_{\Gamma}(E')]/(k[\lk_\Gamma(E')]  \cap (\theta_1', \dotsc, \theta_{d}'))
\]
whose kernel contains $\{[x^F] : F \not \in \Star(E' \smallsetminus E)\}$ and satisfies $\phi(x^F) = x^F$ for all $F \in \lk_{\Gamma}(E')$. Moreover, there is a special l.s.o.p. $\zeta_1, \ldots, \zeta_{d'}$ for $k[\lk_\Gamma(E')]$ such that $(\zeta_1, \ldots, \zeta_{d'}) = k[\lk_\Gamma(E')] \cap (\theta'_1, \ldots, \theta'_{d})$ and, up to reordering, we have $\theta_i|_{\lk_\Gamma(E')} = \zeta_i$, for $1 \leq i \leq b'$.  With this choice of special l.s.o.p., $\phi(L(\Gamma,E)) \subset L(\Gamma,E')$. 
\end{theorem}

\begin{remark}
Theorem~\ref{thm:maps} may be viewed as a functoriality statement for local face modules. Start by fixing the special l.s.o.p. $\theta_1, \ldots, \theta_d$.  Then $L(\Gamma, E)$ is well-defined. For $E' \supset E$ the special l.s.o.p. $\zeta_1, \ldots, \zeta_{d'}$ depends on some choices, but the ideal that it generates does not, nor does the map $\phi \colon L(\Gamma,E) \to L(\Gamma,E')$.  Moreover, for $E'' \supset E'$, one readily checks that the maps $\phi' \colon L(\Gamma, E') \to L(\Gamma, E'')$ and $\phi '' \colon L(\Gamma,E) \to L(\Gamma, E'')$ are independent of all choices and satisfy $\phi'' = \phi' \circ \phi$.  Thus one obtains a functor from the poset of faces of $\Gamma$ that contain $E$ to graded vector spaces, given by $E' \mapsto L(\Gamma, E')$.
\end{remark}

We now give two applications of the above theorems. The first is a monotonicity property for local $h$-vectors.

\begin{theorem}\label{thm:increase} 
Let $E \subset E'$ be faces of $\Gamma$ such that $\sigma(E) = \sigma(E')$. Then $\ell (\Gamma, E) \geq \ell(\Gamma, E')$.
\end{theorem}

\noindent The inequality in Theorem~\ref{thm:increase} is term by term, i.e., $\dim L(\Gamma, E)_i \geq \dim L(\Gamma,E')_i$ for all $i$. The proof is by showing that the map $\phi \colon L(\Gamma, E) \to L(\Gamma,E')$ given by Theorem~\ref{thm:maps} is surjective.

Our second application of the above theorems is to a decades old problem posed by Stanley, who introduced and studied local $h$-vectors in the special case where $E = \emptyset$ and asked for a characterization of triangulations for which they vanish \cite[Problem~4.13]{Stanley92}. This problem remains open, and is of enduring interest \cite[Problem~2.12]{Athanasiadis16}. The extension to the case where $E$ is not empty is particularly relevant for applications to the monodromy conjecture \cite{Igusa78, DenefLoeser98, Stapledon17}. In \cite{LarsonPayneStapledon}, we prove a theorem on the structure of geometric triangulations with vanishing local $h$-vectors that is tailored to this purpose, and we use it to prove the monodromy conjectures for all singularities that are nondegenerate with respect to a simplicial Newton polyhedron.  See Theorems~1.1.1, 1.4.3, and 4.1.3 in loc. cit.  

Here, we apply  Theorem~\ref{thm:resolution} to prove another theorem on the structure of faces in triangulations with vanishing local $h$-vectors.  Let $F \in \lk_\Gamma(E)$ be a face such that $F\sqcup E$ is interior. Following terminology from the monodromy conjecture literature (see, e.g., \cite{LemahieuVanProeyen11}), we say that $F$ is a \emph{pyramid with apex $w \in F$}  if $(F \sqcup E )  \smallsetminus w$ is not interior. Let $$\mathcal{A}_F := \{ w \in F : F \mbox{ is a pyramid with apex } w \}, \mbox{ \ and \ } V_w := \sigma( (F \sqcup E) \smallsetminus w)^c.$$
The elements of $V_w$ correspond to the \emph{base directions} of $F$, i.e., the facets of $2^V$ that contain the base of $F$, when viewed as a pyramid with apex $w$.  We say $F$ is a $U$-pyramid if there is an apex $w \in \mathcal{A}_F$ such that $|V_w|  = 1$. In other words, a $U$-pyramid is a pyramid with a unique base direction, for some choice of apex.

\begin{definition}
Let $F \in \lk_{\Gamma}(E)$ be a face.  An \emph{interior partition} of $F$ is a decomposition
\[
F = F_1 \sqcup F_2 \sqcup \mathcal{A}_F
\]
such that $F_1 \sqcup \mathcal{A}_F \sqcup E$ and $F_2 \sqcup \mathcal{A}_F \sqcup E$ are both interior.
\end{definition}

\begin{theorem}\label{thm:interiornonvanish}
Suppose $\ell(\Gamma,E) = 0$ and $F \in \lk_\Gamma(E)$ has an interior partition $F = F_1 \sqcup F_2 \sqcup \mathcal{A}_F$ such that $|F_1| \leq 2$.  Then $F$ is a $U$-pyramid.
\end{theorem}

\noindent See Remark~\ref{r:specialcase} for a short proof in a special case that illustrates the naturality of the 
$U$-pyramid condition.
The method of proof breaks down when $|F_i| \geq 3$.  See Example~\ref{example:nonrestriction}.

\begin{remark}
The analogous theorem in \cite{LarsonPayneStapledon} requires that the triangulation be geometric and that the interior partition satisfies the additional condition $\sigma(F_2 \sqcup E)^c = \bigcup_{w \in \mathcal{A}_F} V_w$.  But then the hypothesis that $|F_1| \leq 2$ is dropped entirely.  So, even for geometric triangulations, there are cases of Theorem~\ref{thm:interiornonvanish} that are not necessarily covered by \cite[Theorem~4.1.3]{LarsonPayneStapledon}. It should be interesting to look for a common generalization of these vanishing results, and to pursue further progress on Stanley's problem of characterizing triangulations with vanishing local $h$-vector more generally.
\end{remark}

\begin{remark}
To the best of our knowledge, all of the theorems stated in the introduction are new even for regular triangulations. The reader who prefers to do so may safely restrict attention to geometric or even regular triangulations.  However, while the structure results for triangulations with vanishing local $h$-vectors in \cite{dMGPSS20} and \cite{LarsonPayneStapledon} rely on special properties of geometric triangulations, the proofs presented here work equally well for quasi-geometric homology triangulations, and we find it  natural to work in this level of generality.
\end{remark}

We conclude the introduction with an example illustrating the above theorems.

\begin{example} \label{ex:triforce}
Let $\Gamma$ be the \emph{triforce} triangulation, which figures prominently in \cite{dMGPSS20} and in the adventures of hero protagonist Link in the video game series The Legend of Zelda. 

\medskip

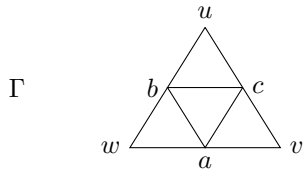
\begin{center}
\begin{tikzpicture}[scale=2]
\draw (0.5,0.8) node[above] { $u$ } -- 
(1,0) node[right] { $v$ } -- 
(0,0) node[left] { $w$ } -- (0.5,0.8);

\draw (0.75,0.4) node[right] { $c$ } -- 
(0.5,0) node[below] { $a$ } -- 
(0.25,0.4) node[left] { $b$ } -- (0.75,0.4) ;

\draw (-.75,0.4) node {$\Gamma$ };

\end{tikzpicture}
\end{center}

Let $x_a := x^{\{ a \}}$, 
$x_b := x^{ \{ b \} }$, 
$x_c := x^{\{ c \}}$, 
$x_u := x^{\{ u \}}$, 
$x_v := x^{\{ v \}}$, 
$x_w := x^{\{ w \}}$. 
Consider first $E = \emptyset$.  The face ring is 
\[
k[\lk_\Gamma(E)] = k[x_{a},x_{b},x_{c},x_{u},x_{v},x_{w}]/(x_{a}x_{u},x_{b}x_{v},x_{c}x_{w},x_{u}x_{v},x_{u}x_{w},x_{v}x_{w}),
\]
and its ideal of interior faces is 
\[
I = (x_ax_b, x_ax_c, x_bx_c).
\]
A special l.s.o.p. is of the form $\theta_1, \theta_2, \theta_3$, with
\[
\supp(\theta_1) = \{ b,c,u\}, \quad \quad \supp(\theta_2) = \{ a,c,v\}, \quad \quad \supp(\theta_3) = \{ a,b,w\},
\]
subject to the condition that the restrictions (of the corresponding affine linear functions) to the face $\{a, b, c\}$ are linearly independent. Our resolution of the local face module $L(\Gamma, E)$ also involves the monomial ideals
\[
\begin{array}{ccc}
I_u = (x_a, x_bx_c), & I_v = (x_b, x_ax_c), & I_w = (x_c, x_ax_b), \\ I_{uv} = (x_a, x_b, x_w), & I_{uw} = (x_a, x_c, x_v), & I_{vw} = (x_b, x_c, x_u).
\end{array}
\]
The resolution given by Theorem~\ref{thm:resolution} is then 
\[
0 \to k[\lk_\Gamma(E)] \xrightarrow{\begin{bsmallmatrix}
\theta_1 \\
-\theta_2 \\
\theta_3
\end{bsmallmatrix}} I_{vw} \oplus I_{uw} \oplus I_{uv}  \xrightarrow{\begin{bsmallmatrix} 
0 & -\theta_3 & -\theta_2\\
-\theta_3 & 0 & \theta_1\\
\theta_2 & \theta_1 & 0\\
 \end{bsmallmatrix} } I_u \oplus I_v \oplus I_w \xrightarrow{\begin{bsmallmatrix}
\theta_1 &
\theta_2 &
\theta_3
\end{bsmallmatrix}} I \to L(\Gamma, E) \to 0
\]
In particular, we have $L(\Gamma, E) \cong I/J$, where 
\[
( \theta_1 \cdot x_a, \theta_2 \cdot x_b, \theta_3 \cdot x_c ) \subset J.
\] 
Since $\theta_1$, $\theta_2$, and $\theta_3$ restrict to linearly independent functions on $\{a, b,c\}$, the elements
$\{ \theta_1 \cdot x_a, \theta_2 \cdot x_b, \theta_3 \cdot x_c \}$
span the 3-dimensional subspace 
$\langle x_a x_b, x_a x_c, x_bx_c \rangle$ 
of $k[\lk_\Gamma(E)]$.  Hence $I = J$ and $L(\Gamma, E) = 0$.

Next, consider $E' = \{c\}$.  Then 
\[
k[\lk_\Gamma(E')] = k[x_a, x_b, x_u, x_v] / (x_ax_u, x_bx_v, x_ux_v).
\]
A special l.s.o.p. is any l.s.o.p. of the form $\zeta_1, \zeta_2$, where $\supp(\zeta_1) \subset \{a, b\}$. The ideal of interior faces in this case is $I' = (x_a, x_b)$, and the resolution given by Theorem~\ref{thm:resolution} is
\[
0 \to k[\lk_\Gamma(E')] \xrightarrow{\begin{bsmallmatrix}
-\zeta_2 \\
\zeta_1 \\
\end{bsmallmatrix}} k[\lk_\Gamma(E')] \oplus I' \xrightarrow{\begin{bsmallmatrix}
\zeta_1 &
\zeta_2
\end{bsmallmatrix}} I' \to L(\Gamma,E') \to 0.
\] 
Note, in particular, that $L(\Gamma,E') \cong I'/J'$, where $J' = (\zeta_1, \zeta_2 x_a, \zeta_2 x_b)$. Thus one sees that $L(\Gamma,E')$ has dimension 1 in degree 1, i.e., $\ell(\Gamma, E') = (0,1,0)$.

Let us now consider Theorem~\ref{thm:maps} in this example.  Let $\theta'_i$ denote the restriction of $\theta_i$ to $k[\Star(E' \smallsetminus E)]$.  Note that $\zeta_1 := \theta'_3$ is supported on $\lk_\Gamma(E')$. Extend $\{ \zeta_1 \}$ to a basis for $k[\lk_\Gamma(E)] \cap (\theta'_1, \theta'_2, \theta'_3)$, e.g., by choosing $\zeta_2$ to be a linear combination of $\theta'_1$ and $\theta'_2$ in which the coefficient of $x_c$ vanishes.  Then $\zeta_1, \zeta_2$ is a special l.s.o.p. for $k[\lk_\Gamma(E')]$, and the map $\phi$ in Theorem 1.4 is given as follows.  First, we set
\[
\phi(x_a) = x_a, \quad \phi(x_b) = x_b, \quad \phi(x_u) = x_u, \quad \phi(x_v) = x_v, \quad \phi(x_w) = 0.
\]
Then, writing $\theta_2 = \lambda_c x_c + \lambda_a x_a + \lambda_v x_v$, with all three coefficients nonzero, we set 
\[
\phi(x_c) = \frac{-1}{\lambda_c} (\lambda_ax_a + \lambda_v x_v).
\]

Note that there is no subset of $\{ \theta_1, \theta_2, \theta_3 \}$ whose restrictions to $k[\lk_\Gamma(E')]$ form an l.s.o.p.  This explains and motivates our two-step process for constructing the map: first restricting to $\Star(E' \smallsetminus E)$ and then intersecting with $k[\lk_\Gamma(E')]$ to produce the special l.s.o.p. that yields the functorial map $\phi \colon L(\Gamma,E) \to L(\Gamma,E')$.

Let also describe how Theorems~\ref{thm:increase} and \ref{thm:interiornonvanish}  manifest in this example.  For Theorem~\ref{thm:interiornonvanish}, observe that the face $F = \{a, b\}$ in $\lk_\Gamma(E')$ has an interior partition $F = \{a\} \sqcup \{b\}$. The proof in this case shows that the classes of both $x_a$ and $x_b$ are nonzero in $L(\Gamma,E')$, for any choice of special l.s.o.p.

Finally, note that $L(\Gamma,E) = 0$ and $L(\Gamma, E') \neq 0$, so there is no surjective map of graded vector space $L(\Gamma,E) \to L(\Gamma,E')$.  In this case, $\sigma(E) \neq \sigma(E')$.  Thus, we see that the hypothesis $\sigma(E) = \sigma(E')$ cannot be dropped in Theorem~\ref{thm:increase}.
\end{example}

\noindent \textbf{Acknowledgments.}  We thank the referees for their helpful comments. The work of ML is supported by an NDSEG fellowship and the work of SP is supported in part by NSF DMS--2001502 and DMS--2053261.

\section{Preliminaries} \label{sec:sr}

We begin by recalling definitions and background results that will be used throughout, following \cite[Chapter~III]{Stanley96} and \cite{Athanasiadis16}. We work over a field $k$. In particular, all rings are commutative $k$-algebras and singular homology is computed with coefficients in $k$.    

\subsection{Triangulations of simplices} \label{sec:triangulations}
In this section only, for the purposes of providing context, we allow that the field $k$ may be finite, and the triangulation $\sigma \colon \Gamma \to 2^V$ is not necessarily quasi-geometric.

\medskip

We recall the notion of a homology triangulation, following \cite{Athanasiadis12}.  A $d$-dimensional simplicial complex $\Gamma$ with trivial reduced homology is a \emph{homology ball} 
of dimension $d$ if there is a subcomplex $\partial \Gamma \subset \Gamma$  such that 
\begin{itemize}
\item $\partial \Gamma$ is a homology sphere of dimension $d -1$,
\item $\lk_\Gamma(F)$ is a homology sphere of dimension $d - |F|$ for $F \not \in \partial \Gamma$.
\item $\lk_\Gamma(F)$ is a homology ball of dimension $d - |F|$ for all nonempty $F \in \partial \Gamma$.
\end{itemize}
The \emph{interior faces} of a homology ball $\Gamma$ are the faces not contained in $\partial \Gamma$.  A  \emph{homology triangulation} of the simplex $2^V$ is a finite simplicial complex $\Gamma$ and a map $\sigma\colon \Gamma \to 2^V$ such that for every non-empty $U \subset V$,
\begin{itemize}
\item the simplicial complex $\Gamma_U := \sigma^{-1}(2^U)$ is a homology ball of dimension $\vert U \vert - 1$.
\item $\sigma^{-1}(U)$ is the set of interior faces of the homology ball $\sigma^{-1}(2^U)$.
\end{itemize}
\noindent Note that the Betti numbers of a simplicial complex, and hence the property of being a homology ball, depend only on the characteristic of the field $k$. Homology triangulations are a special case of the (strong) formal subdivisions of Eulerian posets considered in \cite[\S 7]{Stanley92} and \cite[\S 3]{KatzStapledon16}.

The \emph{carrier} of a face $F \in \Gamma$ is $\sigma(F)$. A homology triangulation $\sigma\colon \Gamma \to 2^V$ is \emph{quasi-geometric} if there is no face $F \in \Gamma$ and $U \subset V$ such that the dimension of $\Gamma_U$ is strictly smaller than the dimension of $F$ and the carrier of every vertex in $F$ is contained in $U$.  A homology triangulation is \emph{geometric} if it can be realized in $\R^n$ as the subdivision of a geometric simplex into geometric simplices. Every geometric homology triangulation is quasi-geometric. 

The local $h$-vector, which we have defined in the introduction as the Hilbert function of the local face module, can be expressed in terms of $h$-vectors of subcomplexes of links of faces in the homology balls $\Gamma_U$:
\begin{equation} \label{eq:localh}
\ell(\Gamma, E) = \sum_{U \supset \sigma(E)} (-1)^{|V| - \vert U \vert} h(\lk_{\Gamma_U}(E)).
\end{equation}
Note that \eqref{eq:localh} makes sense even when $k$ is finite or $\sigma \colon \Gamma \to 2^V$ is not quasi-geometric, and should be taken as the definition of the local $h$-vector in this broader context.

\begin{theorem}[\cite{Stanley92, Athanasiadis12, KatzStapledon16}]\label{t:localproperties}
Let $\sigma \colon \Gamma \to 2^V$ be a homology triangulation, let $E$ be a face of $\Gamma$ and let $d = |V| - |E|$.  Then the local $h$-vector $(\ell_0, \ldots, \ell_d)$ satisfies: \\
\begin{tabular}{lll}
\quad \quad $\bullet$ &  \emph{(symmetry)} &
$\ell_i = \ell_{d-i};$ \\
\quad \quad $\bullet$ & \emph{(non-negativity)} & if $\Gamma$ is quasi-geometric, then $\ell_i \geq 0;$\\
\quad \quad $\bullet$ & \emph{(unimodality)} &
 if $\Gamma$ is regular, then $\ell_0 \leq \ell_1 \leq \cdots \leq \ell_{\lfloor d/2 \rfloor}$.
\end{tabular}
\end{theorem}

\noindent Note that the proof of non-negativity for quasi-geometric triangulations, due to Stanley and Athanasiadis, is via the identification with the Hilbert function of the local face module. It suffices to consider the case where $k$ is infinite, since \eqref{eq:localh} is invariant under field extensions.

\subsection{Face rings and special l.s.o.p.s}\label{ss:sr}

Here, and for the remainder of the paper, the field $k$ is fixed and infinite, and all triangulations are quasi-geometric homology triangulations.

 Given a finite simplicial complex $\Gamma$ with vertex set $V = \{v_1, \ldots, v_n\}$, let $k[\Gamma]$ denote the \emph{face ring}. In other words, for each subset $F \subset V$, let $x^F$ be the corresponding squarefree monomial in the polynomial ring $k[x_1, \ldots, x_n]$, i.e.,
 $
 x^F:= \prod_{v_i \in F} x_i.
 $ 
 Then the face ring is
 \[
 k[\Gamma] := k[x_1, \ldots, x_n] / (x^F : F \mbox{ is not a face in } \Gamma).
 \]
Given a subcomplex $\Gamma'$ of $\Gamma$, we have a natural restriction map $k[\Gamma] \rightarrow  k[\Gamma']$, taking $x^F$ to $x^F$ if $F \in \Gamma'$ and to 0 otherwise. Given $\theta \in k[\Gamma]$, let $\theta|_{\Gamma'}$ denote the image of $\theta$ in $k[\Gamma']$. In particular, each $F$ in $\Gamma$ may be viewed as a subcomplex, and we write $\theta|_F$ for the restriction of $\theta$ to this subcomplex.

Note that $k[\Gamma]$ is graded by degree. By definition, a linear system of parameters (l.s.o.p.) for a finitely generated graded $k$-algebra $R$ of Krull dimension $d$ is a sequence of elements $\theta_1, \ldots, \theta_d$ in $R_1$ such that $R/(\theta_1, \ldots, \theta_d)$ is a finite-dimensional $k$-vector space. If $\Gamma$ is a Cohen-Macaulay complex (i.e., if $k[\Gamma]$ is a Cohen-Macaulay ring) and $\theta_1, \ldots, \theta_d$ is an l.s.o.p.\ for $k[\Gamma]$, then $(\theta_1, \ldots, \theta_d)$ is a regular sequence and the $h$-polynomial of $\Gamma$ is the Hilbert series of $k[\Gamma]/(\theta_1, \ldots, \theta_d)$.  Links of faces in triangulations of simplices are Cohen-Macaulay \cite{Reisner76}. 

Suppose $\Gamma$ has dimension $d-1$, so $k[\Gamma]$ has Krull dimension $d$.  Then a sequence of elements $\theta_1, \ldots, \theta_d$ in $k[\Gamma]_1$ is an l.s.o.p. for $k[\Gamma]$ if and only if the following condition is satisfied \cite[Lemma~2.4(a)]{Stanley96}:

\renewcommand{\labelitemi}{$(*)$}
\begin{itemize}
\item For every face $F \in \Gamma$ (or equivalently, for every facet $F \in \Gamma$), the restrictions $\theta_1|_F, \ldots, \theta_d|_F$ span a vector space of dimension $|F|$. 
\end{itemize}

\noindent  This characterization provides flexibility in constructing l.s.o.p.s in which the linear functions have specified support, where the \emph{support} of $\theta = \sum a_i x_i$ is $\supp(\theta) := \{ v_i : a_i \neq 0 \}$.

\begin{lemma} \label{lemma:lsopexistence}
Let $S_1, \ldots, S_d$ be subsets of the vertices of $\Gamma$.  Then there is an l.s.o.p. $\theta_1, \ldots, \theta_d$ for $k[\Gamma]$ such that $\supp (\theta_i) = S_i$ for $1 \leq i \leq d$ if and only if, for every face $F \in \Gamma$,
\begin{equation} \label{eq:marriageinequality}
| \{ S_i : S_i \cap F \neq \emptyset \}| \geq |F|.
\end{equation}
\end{lemma}

\begin{proof} 
The argument is similar to that given by Stanley in \cite[Corollary~4.4]{Stanley92}. The necessity of \eqref{eq:marriageinequality} follows immediately from (*). We now prove its sufficiency.  Suppose $S_1, \ldots, S_d$ are chosen such that \eqref{eq:marriageinequality} holds for every 
$F \in \Gamma$.    
Let $N = |S_1| + \cdots + |S_d|$, and consider the space $k^N$ 
parametrizing tuples $(\theta_1, \ldots, \theta_d)$
with $\supp (\theta_i) \subset S_i$.
Fix $F = \{v_1, \dotsc, v_k\} \in \Gamma$.
Let $X_F \subset k^N$  parametrize the tuples 
 whose restrictions to $F$ 
 span a vector space of dimension $|F|$.
Note that $X_F$ is Zariski open.  
By Hall's Marriage Theorem, there is a permutation $\sigma \in \mathfrak{S}_d$ such that $v_i \in S_{\sigma(i)}$. If we set $\theta_{\sigma(i)} = x_i$ for $1 \le i \le k$, and $\theta_{\sigma(i)} = 0$ for $i > k$, then $\theta \in X_F$, and hence $X_F$ is nonempty. Also, the subset of $k^N$ where all coordinates are nonzero is Zariski open and nonempty.  Since $k$ is infinite, the intersection of these nonempty Zariski open subsets of $k^N$ is nonempty, and hence there is an l.s.o.p. $\theta_1, \ldots, \theta_d$ with $\supp(\theta_i)= S_i$. 
\end{proof}

Let $\sigma \colon \Gamma \to 2^V$ be a quasi-geometric homology triangulation,   and let $E \in \Gamma$ be a face. 

\begin{definition}[\cite{Stanley92, Athanasiadis12b}] \label{def:special} A linear system of parameters $\theta_1, \dotsc, \theta_{d}$ for $k[\lk_{\Gamma}(E)]$ is \textit{special} if, for each vertex $v \in V$ with $v \not \in \sigma(E)$, there is an element $\theta_v$ of the l.s.o.p. such that $\supp(\theta_v)$ consists of vertices in $\lk_{\Gamma}(E)$  whose carrier contains $v$, and such that $\theta_v \not= \theta_{v'}$ for $v \not= v'$. 
\end{definition}

In other words, after reordering so that $\sigma(E)^c = \{v_1, \ldots, v_b\}$, an l.s.o.p. for $k[\lk_\Gamma(E)]$ is special if we can order it $\theta_1, \ldots, \theta_d$ such that
\[
\supp(\theta_i) \subset \{ w \in \lk_\Gamma(E) : v_i \in \sigma(w)\},
\]
for $1 \leq i \leq b$.  The existence of special l.s.o.p.s is well-known to experts and the proof is similar to Stanley's argument in the case $E = \emptyset$. For completeness, we provide a short proof.

\begin{proposition} \label{prop:speciallsopexistence}
Suppose $k$ is infinite. Let $\sigma\colon \Gamma \to 2^V$ be a quasi-geometric homology triangulation of a simplex, and let $E$ be a face of $\Gamma$. Then there is a special l.s.o.p. for $k[\lk_{\Gamma}(E)]$.
\end{proposition}
\begin{proof}

Let $V = \{v_1, \ldots, v_n\}$.  After renumbering, we may assume that $\sigma(E)^c = \{v_1, \dotsc, v_b\}$. Fix $d = n - \vert E \vert$. Note that $b \leq d$.  We define subsets $S_1, S_2, \dotsc, S_{d}$ of the vertices in $\lk_{\Gamma}(E)$, as follows. For $i \le b$, let $S_i$ be the set of vertices $w$ such that $v_i \in \sigma(w)$. 
For $i > b$, let $S_i$ be the set of all vertices of $\lk_{\Gamma}(E)$. Because $\sigma$ is quasi-geometric, for each face $F$ of $\lk_{\Gamma}(E)$, the union of the sets $\sigma(w) \subset V$, as $w$ ranges over vertices of $E \sqcup F$, has size at least $|E| + |F|$.  It follows that $|\{i \leq b : S_i \cap F \neq \emptyset \}| \geq |F| - (d-b)$.  Since $S_j \cap F \neq \emptyset$ for $j > b$, we conclude that $|\{i : S_i \cap F \neq \emptyset \}| \geq |F|$.  Hence, by Lemma~\ref{lemma:lsopexistence}, there is an l.s.o.p. $\theta_1, \ldots, \theta_d$ for $k[\lk_\Gamma(E)]$ with $\supp(\theta_i) = S_i$. 
\end{proof}

\section{A resolution of the local face module}

In this section, we prove Theorem~\ref{thm:resolution}, giving an explicit resolution of the local face module $L(\Gamma,E)$ by a subcomplex of the Koszul resolution of $k[\lk_{\Gamma}(E)]/(\theta_1, \dotsc, \theta_{d})$.  We continue to use the notation established above.  In particular, $\sigma\colon \Gamma \to 2^V$ is a quasi-geometric homology triangulation of the simplex with vertex set $V = \{v_1, \ldots, v_n\}$. We consider a face $E \in \Gamma$ with $d = n - |E|$ and $b = n - |\sigma(E)|$.  After reordering, we assume $\sigma(E)^c = \{v_{1}, \dotsc, v_{b}\}$. For $S \subset \{v_1, \ldots, v_d\}$, 
we consider the ideal $I_S \subset k[\lk_\Gamma(E)]$ given by
\begin{equation*}
I_S := ( x^F : \, \sigma(F \sqcup E)^c \subset S).  
\end{equation*}

Let $\theta_1, \ldots \theta_d$ be a special l.s.o.p. for $k[\lk_\Gamma(E)]$. We may assume that 
\[
\supp(\theta_i) \subset \{ w \in \lk_\Gamma(E) : v_i \in \sigma(w) \},
\]
for $1 \leq i \leq b$.  For any $v_i \in S$, multiplication by $\theta_i$ gives a map $\lambda_i \colon I_{S} \to I_{S \smallsetminus \{v_i\}}$, and we consider the complex of graded $k[\lk_{\Gamma}(E)]$-modules
 \begin{equation} \label{eq:resolution}
 0  \to k[\lk_{\Gamma}(E)][-d] \to \bigoplus_{ |S|  = d - 1 } I_{S}[-(d-1)] \to \dotsb  \to \bigoplus_{|S| = 1} I_{S}[-1] \to I \to L(\Gamma,E) \to 0,
\end{equation}
in which
 the differential restricted to $I_S$, for $S= \{v_{i_0}, \ldots, v_{i_k}\}$, with $i_0 < \cdots < i_k$, is $\oplus_{j = 0}^k (-1)^j \lambda_{i_j}$.

\begin{example}
If $E$ is an interior face of $\Gamma$ then every l.s.o.p. is special, $I_S = k[\lk_{\Gamma}(E)]$ for all $S$, and \eqref{eq:resolution} is the Koszul resolution of $L(\Gamma,E) = k[\lk_{\Gamma}(E)]/(\theta_1,\ldots,\theta_{d})$.
\end{example}

\begin{proof}[Proof of Theorem \ref{thm:resolution}]
We must show \eqref{eq:resolution} is exact.
 We begin by considering two complexes of $k[\lk_{\Gamma}(E)]$-modules
 studied by Stanley and Athanasiadis. Recall that, for $U \subset V$, we write $\Gamma_U := \sigma^{-1}(2^U)$.
 
Say $U \supset \sigma(E)$ and $U \smallsetminus \sigma(E) = \{v_{i_0}, \ldots, v_{i_k}\}$, with $i_0 < \cdots < i_k$. For $0 \leq j \leq k$, let $\rho_j \colon k[\lk_{\Gamma_U}(E)] \to k[\lk_{\Gamma_{U \smallsetminus \{v_{i_j}\}}}(E)]$ be the restriction map.  The first complex we consider is 
\small
\begin{equation}\label{eq:modcomp}
\begin{tikzcd}
 k[\lk_{\Gamma}(E)]  \arrow[r] & \bigoplus\limits_{\substack{U \supset \sigma(E)  \\ \vert U \vert = n - 1}} k[\lk_{\Gamma_U}(E)] \arrow[r]  &\bigoplus\limits_{\substack{U \supset \sigma(E)  \\ \vert U \vert = n - 2}}  k[\lk_{\Gamma_U}(E)] \arrow[r]  &\cdots \arrow[r] & k[\lk_{\Gamma_{\sigma(E)}}(E)] \arrow[r] & 0,
\end{tikzcd}
\end{equation}
\normalsize
in which the differential restricted to $k[\lk_{\Gamma_U}(E)]$ is $\bigoplus_j (-1)^j \rho_j$.  Next, we consider its quotient by $(\theta_1,\ldots,\theta_{d})$:

\begin{equation}\label{eq:quotientedcomplex}
\begin{tikzcd}
 \frac{k[\lk_{\Gamma}(E)]}{(\theta_1, \dotsc, \theta_{d})} \arrow[r] & \bigoplus\limits_{\mathclap{\substack{U \supset \sigma(E)  \\ \vert U \vert = n-1}}} \frac{k[\lk_{\Gamma_{U}}(E)]}{(\theta_1, \dotsc, \theta_{d})} \arrow[r]  & \bigoplus\limits_{\mathcal{\substack{U \supset \sigma(E)  \\ \vert U \vert = n-2}}} \frac{k[\lk_{\Gamma_{U}}(E)]}{(\theta_1, \dotsc, \theta_{d})} \arrow[r] &\cdots \arrow[r] & \frac{k[\lk_{\Gamma_{\sigma(E)}}(E)]}{(\theta_1, \dotsc, \theta_{d})} \arrow[r]  &0.
 \end{tikzcd}
\end{equation}

For any $U \subset V$, with $U \supset \sigma(E)$, let $S_U$ be defined as
\[
S_U := (U \cap \{v_1, \ldots, v_b\}) \cup \{ v_{b+1}, \ldots, v_{d}\}.
\]
Then $\dim k[\lk_{\Gamma_U}(E)] = |S_U|$ and it follows that the restriction of $\theta_i$ to $\lk_{\Gamma_U}(E)$ is nonzero if and only if $v_i \in S_U$. Furthermore, $\{ \theta_i|_{\lk_{\Gamma_U}(E)} : v_i \in S_U \}$ is a special l.s.o.p. for $k[\lk_{\Gamma_{U}}(E)]$.  Stanley and Athanasiadis proved that both \eqref{eq:modcomp} and \eqref{eq:quotientedcomplex} are exact, and the kernel of the first arrow in \eqref{eq:quotientedcomplex} is $L(\Gamma,E)$. (We will recall the proofs below.) Using the additivity of Hilbert functions in exact sequences, they deduced that the Hilbert function of $L(\Gamma, E)$ satisfies \eqref{eq:localh} \cite{Stanley92, Athanasiadis12}.

With the goal of proving that \eqref{eq:resolution} is exact, we take Koszul resolutions of each term in \eqref{eq:quotientedcomplex} to build a double complex of $k[\lk_{\Gamma}(E)]$-modules. Since $k[\lk_{\Gamma_{U}}(E)]$ is Cohen-Macauley, the special l.s.o.p.  $\{ \theta_i|_{\lk_{\Gamma_U}(E)} : v_i \in S_U \}$ is a regular sequence. Hence the corresponding Koszul  complex $K^{\bullet}_U$
\begin{center}
\begin{tikzcd}[column sep = small]
0 \arrow[r] &k[\lk_{\Gamma_U}(E)]_{S_U} \arrow[r] &\bigoplus\limits_{\mathclap{\substack{S \subset S_U \\ \vert S \vert = \vert S_U \vert - 1}}} k[\lk_{\Gamma_U}(E)]_S \arrow[r]& \cdots \arrow[r] &\bigoplus\limits_{\mathclap{\substack{S\subset S_U\\ \vert S \vert = 1}}} k[\lk_{\Gamma_U}(E)]_S \arrow[r]  &k[\lk_{\Gamma_U}(E)] \arrow[r]& \frac{k[\lk_{\Gamma_U}(E)]}{(\theta_1, \dotsc, \theta_{d})}  \arrow[r] &0,
\end{tikzcd}
\end{center}
is exact.  Here, for a graded module $M$ and a finite set $S$, we write $M_{S} := M[-|S|]$. Replacing each term in \eqref{eq:quotientedcomplex} with its corresponding Koszul resolution, gives a complex of complexes
\begin{equation}\label{eq:koszul}
\begin{tikzcd}
K_{V}^{\bullet} \arrow[r] & \bigoplus\limits_{\mathclap{\substack{U \supset \sigma(E) \\ \vert U \vert = n - 1}}} K_U^{\bullet} \arrow[r] & \bigoplus\limits_{\mathclap{\substack{U \supset \sigma(E) \\ \vert U \vert = n - 2}}} K_U^{\bullet} \arrow[r] & \cdots \arrow[r] & K_{\sigma(E)}^{\bullet} \arrow[r] & 0,
\end{tikzcd}
\end{equation}
which may be expanded as the commuting double complex shown in Figure~\ref{fig:doublecx}.
\begin{figure}[h!] 
\begin{center}
 \begin{tikzcd}[column sep = tiny]
 0  &0  &0 & \cdots  & 0  \\
 \frac{k[\lk_{\Gamma}(E)]}{(\theta_1, \dotsc, \theta_{d})} \arrow[r] \arrow[u] & \bigoplus\limits_{\mathclap{\substack{U \supset \sigma(E)  \\ \vert U \vert = n-1}}} \frac{k[\lk_{\Gamma_{U}}(E)]}{(\theta_1, \dotsc, \theta_{d})} \arrow[r] \arrow[u] & \bigoplus\limits_{\mathcal{\substack{U \supset \sigma(E)  \\ \vert U \vert = n-2}}} \frac{k[\lk_{\Gamma_{U}}(E)]}{(\theta_1, \dotsc, \theta_{d})} \arrow[r] \arrow[u] &\cdots \arrow[r] & \frac{k[\lk_{\Gamma_{\sigma(E)}}(E)]}{(\theta_1, \dotsc, \theta_{d})} \arrow[r] \arrow[u] &0 \\
 k[\lk_{\Gamma}(E)]  \arrow[r] \arrow[u] & \bigoplus\limits_{\substack{U \supset \sigma(E)  \\ \vert U \vert = n - 1}} k[\lk_{\Gamma_U}(E)] \arrow[r] \arrow[u] &\bigoplus\limits_{\substack{U \supset \sigma(E)  \\ \vert U \vert = n - 2}}  k[\lk_{\Gamma_U}(E)] \arrow[r] \arrow[u] &\cdots \arrow[r] & k[\lk_{\Gamma_{\sigma(E)}}(E)] \arrow[r] \arrow[u] & 0 \\
\bigoplus\limits_{\mathclap{\substack{\vert S \vert = 1}}}k[\lk_{\Gamma}(E)]_S \arrow[r] \arrow[u] & \bigoplus\limits_{\substack{U \supset \sigma(E)  \\ \vert U \vert = n - 1}} \quad  \bigoplus\limits_{\mathclap{\substack{S \subset S_U \\ \vert S \vert = 1}}} k[\lk_{\Gamma_U}(E)]_S \arrow[r] \arrow[u] &\bigoplus\limits_{\substack{U \supset \sigma(E)  \\ \vert U \vert = n - 2}} \quad  \bigoplus\limits_{\mathclap{\substack{S \subset S_U \\ \vert S \vert = 1}}} k[\lk_{\Gamma_U}(E)]_S \arrow[r] \arrow[u] &\cdots \arrow[r] & \bigoplus\limits_{\mathclap{\substack{S \subset S_{\sigma(E)} \\ \vert S \vert =  1}}} k[\lk_{\Gamma_{\sigma(E)}}(E)]_S \arrow[r] \arrow[u] & 0 \\
\bigoplus\limits_{\mathclap{\substack{\vert S \vert = 2}}}k[\lk_{\Gamma}(E)]_S \arrow[r] \arrow[u] & \bigoplus\limits_{\substack{U \supset \sigma(E)  \\ \vert U \vert = n - 1}} \quad  \bigoplus\limits_{\mathclap{\substack{S \subset S_U \\ \vert S \vert = 2}}} k[\lk_{\Gamma_U}(E)]_S \arrow[r] \arrow[u] & \bigoplus\limits_{\substack{U \supset \sigma(E)  \\ \vert U \vert = n - 2}} \quad  \bigoplus\limits_{\mathclap{\substack{S \subset S_U \\ \vert S \vert = 2}}} k[\lk_{\Gamma_U}(E)]_S \arrow[r] \arrow[u] &\cdots \arrow[r] & \bigoplus\limits_{\mathclap{\substack{S \subset S_{\sigma(E)} \\ \vert S \vert = 2}}} k[\lk_{\Gamma_{\sigma(E)}}(E)]_S \arrow[r] \arrow[u] & 0. \\
\vdots  \arrow[u] & \vdots \arrow[u]  &\vdots \arrow[u] &\cdots  & \vdots \arrow[u]  \\
\bigoplus\limits_{\mathclap{\substack{\vert S \vert = d -1}}}k[\lk_{\Gamma}(E)]_S \arrow[r] \arrow[u] & \bigoplus\limits_{\mathclap{\substack{U \supset \sigma(E) \\ \vert U \vert = n - 1}}}k[\lk_{\Gamma}(E)]_{S_U} \arrow[u]\arrow[r] &0 \arrow[u] \\
k[\lk_{\Gamma}(E)]_{\{v_1, \ldots, v_d\}} \arrow[u] \arrow[r] &0 \arrow[u] \\
0 \arrow[u]
\end{tikzcd}
\end{center}
\vspace{-15 pt}
\caption{The double complex obtained by taking the Koszul resolution of \eqref{eq:quotientedcomplex}.} \label{fig:doublecx}
\end{figure}
The columns of this complex are exact by construction. We claim that the rows are also exact, and prove this using ideas from \cite[Theorem~4.6]{Stanley92}. First, we show that all rows except for the top row are exact. Choose a subset $S$ of $\{v_1, \ldots, v_d\}$, and consider the piece of the complex indexed by $S$:
\begin{equation}\label{eq:generalS}
\begin{tikzcd}[column sep = small]
 k[\lk_{\Gamma}(E)]_S  \arrow[r]  & \bigoplus\limits_{\substack{ S \subset S_U   \\ \vert U \vert = n - 1}} k[\lk_{\Gamma_U}(E)]_S \arrow[r]  &\bigoplus\limits_{\substack{S \subset S_U \\ \vert U \vert = n - 2}}  k[\lk_{\Gamma_U}(E)]_S \arrow[r]  &\cdots 
 \arrow[r]  & 0.
\end{tikzcd}
\end{equation}

When $S = \emptyset$, we obtain (\ref{eq:modcomp}). Observe that the complex (\ref{eq:generalS}) is multigraded by $\mathbb{N}^m$, where $m$ is the number of vertices of $\lk_{\Gamma}(E)$. Explicitly, $\deg x_1^{\alpha_1} \dotsb x_m^{\alpha_m} = (\alpha_1, \dotsc, \alpha_m)$. Therefore it suffices to show exactness on graded pieces. Fix $\alpha = (\alpha_1, \dotsc, \alpha_m)$. By the definition of the face ring, every term of (\ref{eq:generalS}) will have $0$ in the graded piece corresponding to $\alpha$ unless the set of vertices with $\alpha_i \not= 0$ forms a face $F$, in which case the $\alpha$-graded part can be identified with the augmented cochain complex of a simplex, indexed by all $U$ that contain $\sigma(E) \cup \sigma(F) \cup S$, and hence is exact.

We now recall the proof that the top row of the double complex, \eqref{eq:quotientedcomplex}, is exact.
\begin{center}
\begin{tikzcd}
 \frac{k[\lk_{\Gamma}(E)]}{(\theta_1, \dotsc, \theta_{d})} \arrow[r]  & \bigoplus\limits_{\mathclap{\substack{U \supset \sigma(E)  \\ \vert U \vert = n-1}}} \frac{k[\lk_{\Gamma_{U}}(E)]}{(\theta_1, \dotsc, \theta_{d})} \arrow[r] & \bigoplus\limits_{\mathcal{\substack{U \supset \sigma(E)  \\ \vert U \vert = n-2}}} \frac{k[\lk_{\Gamma_{U}}(E)]}{(\theta_1, \dotsc, \theta_{d})} \arrow[r]  &\cdots \arrow[r] & \frac{k[\lk_{\Gamma_{\sigma(E)}}(E)]}{(\theta_1, \dotsc, \theta_{d})} \arrow[r]  &0
\end{tikzcd}
\end{center}
The proof involves showing that the quotients of (\ref{eq:modcomp}) by $(\theta_{d}, \dotsc, \theta_{d - (r-1)})$ is exact by induction on $r$. The case of $r = 0$ is the exactness of the second row. 

Now assume that (\ref{eq:modcomp}) remains exact after quotienting by $(\theta_{d}, \dotsc, \theta_{d - (r-1)})$. Let $C^i$ denote the $i$th term of (\ref{eq:modcomp}) tensored with $k[\lk_{\Gamma}(E)]/(\theta_{d}, \dotsc, \theta_{d - (r - 1)})$. By the induction hypothesis, we have an exact sequence
\begin{equation*}
C^{\bullet}\colon \enskip  C^0 \to C^1 \to \dotsb \to C^{b} \to 0.
\end{equation*} 
Set $m = d - r$. Recall that 
$\theta_i = 0 \in k[\lk_{\Gamma_{U}}(E)]$ if $v_i \notin S_U$, and that $\{ \theta_i|_{\lk_{\Gamma_U}(E)} : v_i \in S_U \}$ is a special l.s.o.p. for $k[\lk_{\Gamma_{U}}(E)]$. 
Also, for $\sigma(E) \subset U$, $v_m \notin S_U$ if and only if $v_m \notin U$. Hence, we have an exact sequence

\begin{equation} \label{eq:multbym}
0 \to B^{\bullet} \to C^{\bullet} \xrightarrow{\theta_{m}} C^{\bullet} \to C^{\bullet}/(\theta_{m}) \to 0,
\end{equation}
where 
\[
B^{i} = \bigoplus_{\substack{U \supset \sigma(E), \enskip \vert U \vert = n-i \\ v_m \not \in U}}k[\lk_{\Gamma_U}(E)]/(\theta_{d}, \dotsc, \theta_{m + 1}).
\]
For example, when $m > b$, $v_m \in \sigma(E)$ and $B^{\bullet} = 0$.
Up to signs and a degree shift, we can then identify $ B^{\bullet}$ with the complex (\ref{eq:modcomp}) for $\Gamma|_{\{v_m\}^c}$ quotiented by $(\theta_{d}, \dotsc, \theta_{m + 1})$. Then $ B^{\bullet}$ is exact by the induction hypothesis applied to $\Gamma|_{\{v_m\}^c}$. By breaking (\ref{eq:multbym}) up into two short exact sequences we see that $H^i(C^{\bullet}/(\theta_{m})) \cong H^{i + 2}(B^{\bullet}) = 0$ as desired.

Now that we know the exactness of (\ref{eq:koszul}), let 
\begin{equation*} \begin{split}
A^{\bullet} &= \ker \Bigg( K_{V}^{\bullet} \to \bigoplus\limits_{\mathclap{\substack{U \supset \sigma(E) \\ \vert U \vert = n - 1}}} K_U^{\bullet} \Bigg).
\end{split}\end{equation*}
Then, by construction, we have an exact sequence of complexes 
\begin{equation*}
\begin{tikzcd}
0 \arrow[r] & A^{\bullet} \arrow[r] &K_{V}^{\bullet} \arrow[r] & \bigoplus\limits_{\mathclap{\substack{U \supset \sigma(E) \\ \vert U \vert = n - 1}}} K_U^{\bullet} \arrow[r] & \bigoplus\limits_{\mathclap{\substack{U \supset \sigma(E) \\ \vert U \vert = n - 2}}} K_U^{\bullet} \arrow[r] & \cdots \arrow[r] & K_{\sigma(E)}^{\bullet} \arrow[r] & 0.
\end{tikzcd}
\end{equation*}
As above, we repeatedly apply the long exact sequence on cohomology to see that $A^{\bullet}$ is exact. 
We may then identify $A^{\bullet}$ with the exact sequence
$$0  \to k[\lk_{\Gamma}(E)][-n] \to \oplus_{ |S|  = d - 1 } I_{S}[-(n-1)] \to \dotsb  \to \oplus_{|S| = 1} I_{S}[-1] \to I \to A^0 \to 0.$$ Since $I$ surjects onto $A^0$ and $A^0 \subset k[\lk_{\Gamma}(E)]/(\theta_1, \dotsc, \theta_{d})$, we conclude that $A^0 = L(\Gamma,E)$, as required.
\end{proof}

\begin{remark}\label{r:specialcase}
Let $\sigma\colon \Gamma \to 2^V$ be a quasi-geometric homology triangulation of a simplex, and let $E$ be a face of $\Gamma$. Let $F \in \lk_\Gamma(E)$ such that $F \sqcup E$ is interior, and suppose that $F = \mathcal{A}_F$ is an interior partition of $F$, i.e., with $F_1 = F_2 = \emptyset$. Suppose that $F$ is not a $U$-pyramid.  
By Corollary \ref{cor:presentation}, $J$ is generated by elements of the form $\theta_i \cdot x^{F}$ for $F \sqcup E$ interior or $\theta_j \cdot x^G$ for some $G$ with $\sigma(G \sqcup E) = \{v_j\}^c$. Because $F$ is not a $U$-pyramid, no monomial appearing in any of these generators divides $x^F$, so $x^F$ is nonzero in $L(\Gamma, E)$. 
This proves Theorem~\ref{thm:interiornonvanish} in the special case when $F_1 = F_2 = \emptyset$. 
\end{remark}

\section{Functorial properties of local face modules}

In this section, we prove Theorem~\ref{thm:maps}, giving natural maps between local face modules.  
Consider a quasi-geometric homology triangulation $\sigma\colon \Gamma \to 2^V$, and let $E \subset E'$ be faces of $\Gamma$.

\begin{lemma}\label{lemma:lsop}
Let $R$ be a graded $k$-algebra with $R_0 = k$. Let $\{ \theta_1, \dotsc, \theta_n \}$ be an l.s.o.p. for $R[x_1, \dotsc, x_m]$, where each $x_j$ has  degree $1$. Then there is a unique graded $R$-algebra isomorphism 
\[
\phi \colon R[x_1, \dotsc, x_m]/(\theta_1, \dotsc, \theta_n) \rightarrow R/R \cap (\theta_1, \dotsc, \theta_n).
\]
Moreover, any $k$-basis for $R_1 \cap (\theta_1, \dotsc, \theta_n)$ is an l.s.o.p. for $R$ and generates $R \cap (\theta_1, \dotsc, \theta_n)$.  
\end{lemma}

\begin{proof}
Consider the exact sequence of $k$-linear maps
\[
0 \rightarrow R_1 \rightarrow R[x_1, \dotsc, x_m]_1 \rightarrow (x_1,\dotsc, x_m)_1 \rightarrow 0,
\]
where the right hand map takes $r + \sum_i \alpha_i x_i$ to $\sum_i \alpha_i x_i$, for any $r \in R_1$ and $\alpha_i \in k$. This restricts to an exact sequence of $k$-linear maps
\[
0 \rightarrow R_1 \cap (\theta_1, \dotsc, \theta_n)_1 \rightarrow (\theta_1, \dotsc, \theta_n)_1 \rightarrow (x_1,\dotsc, x_m)_1 \rightarrow 0,
\]
where the surjectivity of the right-hand map follows from the fact that $\theta_1, \dotsc, \theta_n$ is an l.s.o.p. Hence, for $1 \le i \le m$, we can write $x_i = r_i + s_i$, for some $r_i \in R_1$ and $s_i \in (\theta_1, \dotsc, \theta_n)_1$. For any $R$-algebra map $\phi \colon R[x_1, \dotsc, x_m]/(\theta_1, \dotsc, \theta_n) \rightarrow R/R \cap (\theta_1, \dotsc, \theta_n),$ we must have that $\phi(x_i) = r_i$, so there is a unique such map. On the other hand, the $R$-algebra homomorphism defined by $\phi(x_i) = r_i$ is well-defined, since if $x_i = r_i' + s_i'$, for some $r_i' \in R_1$ and $s_i' \in (\theta_1, \dotsc, \theta_n)_1$, then $r_i - r_i' \in R_1 \cap (\theta_1, \dotsc, \theta_n)_1$. Note that the unique $R$-algebra homomorphism from $R/R \cap (\theta_1, \dotsc, \theta_n)$ to  $R[x_1, \dotsc, x_m]/(\theta_1, \dotsc, \theta_n)$ is the inverse of $\phi$.

Since $\phi$ is an isomorphism and factors through $R/(R_1 \cap (\theta_1, \dotsc, \theta_n)_1)$, we conclude that the $R$-ideal $R \cap (\theta_1, \dotsc, \theta_n)$ is generated in degree $1$ and hence any $k$-basis for $R_1 \cap (\theta_1, \dotsc, \theta_n)$ is an l.s.o.p. for $R$.
\end{proof}

\begin{proof}[Proof of Theorem~\ref{thm:maps}]
Note that $\Star(E' \smallsetminus E)$ is the join of $E' \smallsetminus E$ with $\lk_{\Gamma}(E')$. The face ring $k[\Star(E' \smallsetminus E)]$ is therefore a polynomial ring over $k[\lk_\Gamma(E')]$. Its Krull dimension is equal to $d = \dim k[\lk_\Gamma(E)]$, and hence the restrictions $\theta'_1, \ldots, \theta'_d$ form an l.s.o.p., where $\theta'_i := \theta_i|_{\Star(E' \smallsetminus E)}$.  By Lemma~\ref{lemma:lsop}, there is a unique graded $k[\lk_\Gamma(E')]$-algebra homomorphism $k[\Star(E' \smallsetminus E)]/(\theta'_1, \ldots, \theta'_d) \to k[\lk_\Gamma(E')]/ (k[\lk_\Gamma(E')] \cap (\theta'_1, \ldots, \theta'_d))$, which lifts to the unique homomorphism $\phi$ in the statement of the theorem. It remains to construct a special l.s.o.p. for $k[\lk_\Gamma(E')]$ with the specified properties.

After reordering, we may assume that
\[
\sigma(E)^c = \{ v_1, \ldots, v_b\}, \quad \supp(\theta_i) \subset \{ w : v_i \in \sigma(w) \}, \ \mbox{ for } 1 \leq i \leq b, \quad \mbox{ and } \sigma(E')^c = \{v_1, \ldots, v_{b'}\}.
\]
Note, in particular, that $\theta'_i$ is supported on vertices in the link of $E'$, for $1 \leq i \leq b'$. By Lemma~\ref{lemma:lsop}, any $k$-basis for $k[\lk_\Gamma(E')] \cap (\theta'_1, \ldots, \theta'_d)$ is an l.s.o.p. for $k[\lk_\Gamma(E')]$.  Set $\zeta_i = \theta_i|_{\lk_\Gamma(E')}$, for $1 \leq i \leq b'$, and note that $\{\zeta_1, \ldots, \zeta_{b'}\}$ is linearly independent.  Extending this independent set to a basis produces a special l.s.o.p. for $k[\lk_\Gamma(E')]$.  It remains to verify that $\phi(L(\Gamma, E)) \subset L(\Gamma, E')$. Let $F \in \lk_{\Gamma}(E)$ be a face with $F \sqcup E$ interior. If $F$ is not in $\Star(E' \smallsetminus E)$, then $\phi(x^F) = 0$. Otherwise, $F$ can be written uniquely as the join of possibly empty faces $F_1 \subset E' \smallsetminus E$ and $F_2 \in \lk_{\Gamma}(E')$. Then $F_2 \sqcup E'$ is interior, and $\phi(x^F) = \phi(x^{F_1})x^{F_2} \in (x^{F_2})$. Hence $\phi(x^F) \in L(\Gamma,E')$, as required.
\end{proof}

\begin{proof}[Proof of Theorem \ref{thm:increase}]
Let $E \subset E'$ be faces of a 
quasi-geometric homology triangulation $\Gamma$ of a simplex, and assume that $\sigma(E) = \sigma(E')$. It is enough to show that the induced map $\phi \colon L(\Gamma, E) \to L(\Gamma,E')$ given by Theorem~\ref{thm:maps} is surjective. Note that $L(\Gamma,E')$ is generated by the monomials $x^F$ such that $F \in \lk_\Gamma(E')$ and $F \sqcup E'$ is interior.  If $F$ is such a face, then it is also in the link of $E$ and, since $\sigma(E) = \sigma(E')$, the face $(F \sqcup E) < (F \sqcup E')$ is also interior.  Then $\phi(x^F) = x^F$, and the theorem follows. \end{proof}

\section{Restrictions of local face modules}
In this section, we use the resolution found in Theorem~\ref{thm:resolution} to show that the vanishing of a local face module $L(\Gamma, E)$ implies the vanishing of a \emph{restricted local face module} $L(\Gamma, \mathcal{A}_F \sqcup E)|_{F_1 \sqcup F_2},$ for certain interior partitions $F_1 \sqcup F_2  \sqcup  \mathcal{A}_F$.  We then develop algebraic arguments, inspired by ideas from \cite{dMGPSS20}, to show that $F$ being a $U$-pyramid is necessary for the vanishing of the restricted local face module when $\vert F_1 \vert \le 2$ and thus prove Theorem~\ref{thm:interiornonvanish}. 

We use the notation introduced in the introduction. Let $\Delta$ be a subcomplex of $\lk_{\Gamma}(E)$. For any $k[\lk_{\Gamma}(E)]$-module $M$, the \emph{restriction} of $M$ to $\Delta$ is $M|_{\Delta} := M \otimes_{k[\lk_{\Gamma}(E)]} k[\Delta]$, where $k[\Delta]$ is a $k[\lk_{\Gamma}(E)]$-module via the restriction map. By the resolution of $L(\Gamma,E)$ in  Theorem~\ref{thm:resolution} and the right exactness of tensoring with $k[\Delta]$, we have an exact sequence 
\begin{equation}\label{e:restrict}
 \bigoplus_{|S| = 1} I_{S}|_\Delta[-1] \to I|_\Delta \to L(\Gamma,E)|_\Delta \to 0.
\end{equation}
Recall from Corollary~\ref{cor:presentation} that $L(\Gamma, E) \cong I/J$, where $J$ is the ideal generated by $\{\theta_i x^{F} : F \sqcup E \mbox{ is interior}\}$ and $\{\theta_{j} x^G : \sigma(G \sqcup E) = \{v_j\}^c \}$.  Hence, $L(\Gamma,E)|_\Delta \cong I|_\Delta/J|_\Delta$, where $I|_\Delta,J|_\Delta$ are the $k[\Delta]$-ideals
\begin{equation}\label{eq:Ires}
I|_\Delta = (x^H :  H \subset \Delta, \sigma(H \sqcup E) = V),
\end{equation}
\begin{equation}\label{eq:Jres}
J|_\Delta = (\theta_1|_\Delta,\ldots,\theta_{d}|_\Delta) \cdot I|_\Delta + (\theta_{j}|_{\Delta} x^G: \enskip G \subset \Delta, \enskip \sigma(G \sqcup E) = \{v_j\}^c).    
\end{equation}
For example, if $F$ is a face of $\lk_{\Gamma}(E)$, then $k[F]$ is a polynomial ring with variables indexed by the vertices of $F$, and 
$L(\Gamma, E)|_{F}$ is identified with  a quotient of ideals in this polynomial ring.

\begin{lemma}\label{lem:restriction}
Let $\sigma\colon \Gamma \to 2^V$ be a quasi-geometric homology triangulation of a simplex, and let $E$ be a face of $\Gamma$. Let $F \in \lk_{\Gamma}(E)$ be a face with $F \sqcup E$ interior. Assume that $F$ is not a $U$-pyramid. 
Then there is a surjective graded $k[F]$-module homomorphism
$$L(\Gamma, E)|_{F} \to L(\Gamma, \mathcal{A}_F \sqcup E)|_{F \smallsetminus \mathcal{A}_F}[-\vert \mathcal{A}_F \vert],$$
where the second term is a $k[F]$-module via the restriction map  $k[F] \mapsto k[F \smallsetminus \mathcal{A}_F]$.
\end{lemma}

\begin{proof}
If $\Delta$ is a subcomplex of $\lk_{\Gamma}(E)$ contained in the closed star of $\mathcal{A}_F$, then $x^{\mathcal{A}_F}$ is a non-zero divisor in $k[\Delta]$. In particular, $x^{\mathcal{A}_F}$ is a non-zero divisor in $k[F]$ (this is also clear since $k[F]$ is a polynomial ring). Note that every face of $F$ with carrier codimension at most $1$ contains $\mathcal{A}_F$. Thus $I|_F = x^{\mathcal{A}_F} \cdot M$ and $J|_F = x^{\mathcal{A}_F} \cdot N$, where 
$M$ and $N$ are the ideals in $k[F]$ 
\[
M = (x^H: \enskip H \subset F \smallsetminus \mathcal{A}_F, \enskip \sigma(H \sqcup \mathcal{A}_F \sqcup E) = V),
\]
\[
N = (\theta_1|_F,\ldots,\theta_{d}|_F) \cdot M + (\theta_{j}|_{F} x^G: \enskip G \subset F \smallsetminus \mathcal{A}_F, \enskip \sigma(G \sqcup \mathcal{A}_F \sqcup E) = \{v_j\}^c).
\]
Then we have surjective graded $k[F]$-module homomorphisms
\[
I|_F/J|_F \rightarrow M /N [-|\mathcal{A}_F|] \rightarrow  M|_{F \smallsetminus \mathcal{A}_F}/N|_{F \smallsetminus \mathcal{A}_F}[-|\mathcal{A}_F|], 
\]
where the first map is the isomorphism taking $x^{\mathcal{A}_F} x^H \mapsto x^H$ and the second map is restriction. Finally the right hand term can be identified with $L(\Gamma, \mathcal{A}_F \sqcup E)|_{F \smallsetminus \mathcal{A}_F}[-\vert \mathcal{A}_F \vert]$.
\end{proof}

We will derive Theorem~\ref{thm:interiornonvanish} from  the following more technical statement.

\begin{theorem}\label{thm:localizedcase}
Let $\sigma\colon \Gamma \to 2^V$ be a quasi-geometric homology triangulation, and let $E$ be a face. Let $F \in \lk_{\Gamma}(E)$ be a face with $F \sqcup E$ interior. Suppose $\mathcal{A}_F = \emptyset$ and $F$ admits an interior partition $F = F_1 \sqcup F_2$. Assume that $F$ has no faces $G$ with $G \sqcup E$ interior and $ \vert G \vert < \vert F_1 \vert$. If $\vert F_1 \vert \le 2$, then $L(\Gamma, E)|_{F}$ is non-zero in degree $\vert F_1 \vert$. 
\end{theorem}

\begin{example}\label{example:nonrestriction}
The conclusion of Theorem~\ref{thm:localizedcase} can fail when $\vert F_1 \vert \ge 3$, even for $\mathcal{A}_F = E = \emptyset$.  Consider a geometric triangulation $\sigma \colon \Gamma \to 2^V$, where $V = \{v_1, \ldots, v_6 \}$ with a face $F = \{w_1, \ldots, w_6\}$ such that 
\[
\begin{array}{lll}
\sigma(w_1) = \{ v_1, v_3, v_6 \} & \sigma(w_2) = \{ v_1, v_4, v_5 \} & \sigma(w_3) = \{ v_2, v_3, v_5 \} \\
\sigma(w_4) = \{ v_2, v_4, v_6 \} & \sigma(w_5) = \{ v_3, v_4, v_5 \} & \sigma(w_6) = \{ v_3, v_5, v_6 \}
\end{array}
\]
Then $\mathcal{A}_F = \emptyset$, and $F$ admits an interior partition given by $F_1 = \{w_1, w_4, w_5\}$, $F_2 = \{w_2, w_3, w_6\}$. Then \eqref{e:restrict} gives generators and relations for $L(\Gamma, \emptyset)|_{F}$, and a linear algebra computation shows that $L(\Gamma, \emptyset)|_{F} = 0$. 
\end{example}

Before proceeding with the proof of Theorem~\ref{thm:localizedcase}, we show how Theorem~\ref{thm:interiornonvanish} follows from it. 

\begin{proof}[Proof of Theorem~\ref{thm:interiornonvanish}]
We may assume that $F = F_1' \sqcup F_2' \sqcup \mathcal{A}_F$ is an interior partition of $F$ with $\vert F_1' \vert$ minimal among all possible interior partitions of $F$. In particular, if $\vert F_1' \vert = 2$, then there is no vertex $v \in F  \smallsetminus \mathcal{A}_F$ such that $\{v\} \sqcup \mathcal{A}_F \sqcup E$ is interior, as then $\{v\} \sqcup (F_1' \sqcup F_2' \smallsetminus \{v\}) \sqcup \mathcal{A}_F$ would be an interior partition. Hence there are no faces $G$ of $F \smallsetminus \mathcal{A}_F$ with $G \sqcup \mathcal{A}_F \sqcup E$ interior and with cardinality smaller than $\vert F_1' \vert$. 
By Theorem \ref{thm:localizedcase}, $L(\Gamma, \mathcal{A}_F \sqcup E)|_{F_1' \sqcup F_2'}$ is non-zero in degree $\vert F_1' \vert$. Then, by Lemma~\ref{lem:restriction},  $L(\Gamma, E)$ is nonzero in degree $\vert F_1' \vert + \vert \mathcal{A}_F \vert$.
\end{proof}

We now proceed with the proof of Theorem~\ref{thm:localizedcase}. We begin with a series of three lemmas. Inspired by the results of \cite{dMGPSS20} in the case $E = \emptyset$, we consider the \emph{internal edge graph} of a subcomplex $\Delta \subset \lk_{\Gamma}(E)$. This is the graph contained in the $1$-skeleton of $\lk_{\Gamma}(E)$ consisting of edges $e \subset \Delta$ with $e \sqcup E$ interior. 

\begin{lemma}\label{lem:intedge}
Assume $\sigma(E)$ has codimension at least $2$. Let $\Delta$ be a subcomplex of $\lk_{\Gamma}(E)$, and assume $\Delta$ has no vertices $v$ with $\{v\} \sqcup E$ interior. If $L(\Gamma, E)|_{\Delta}$ is zero in degree $2$, then each connected component of the internal edge graph of $\Delta$ satisfies one of the following.
\begin{enumerate}
\item The component is a tree, and it has at most one  vertex $v$ with $\{v\} \sqcup E$ having carrier codimension more than $1$. 
\item The component has a unique cycle, and the carrier codimension of $\{w\} \sqcup E$ is equal to $1$ for every vertex $w$ in the component. 
\end{enumerate}
\end{lemma}

\begin{proof}
From \eqref{e:restrict}, we have the following exact sequence for the degree $2$ part of $L(\Gamma, E)|_{\Delta}$.
$$ \bigoplus_{\vert S \vert = 1} (I_S)_1\otimes_{k[\lk_{\Gamma}(E)]} k[\Delta]  \to I_2\otimes_{k[\lk_{\Gamma}(E)]} k[\Delta] \to (L(\Gamma, E)|_{\Delta})_2 \to 0.$$
Because $(L(\Gamma, E)|_{\Delta})_2 = 0$, the first map in the above complex is surjective.  As $\Delta$ has no vertices $v$ with $\{v\} \sqcup E$ interior, we see that
\begin{equation}\label{eq:intedge}
(x^{e}: e \subset \Delta, \enskip e \sqcup E \text{ is interior })_2 = (x^{\{v\}} \theta_{i}: v \subset \Delta, \enskip \sigma(\{v\} \sqcup E) = \{v_i\}^c)_2.
\end{equation}
Thus the number of edges $e$ with $e \sqcup E$ interior is less than or equal to the number of vertices $w$ with the carrier codimension of $\{w\} \sqcup E$ equal to $1$. If $\sigma(\{v\} \sqcup E) = \{v_i\}^c$ and $\theta_{i} = \sum_{w_j} a_{i,j}x^{\{w_j\}}$, then 
$$x^{\{v\}} \theta_{i} = \sum_{\{v, w_j\} \sqcup E \text{ interior }} a_{i,j}x^{\{v, w_j\}}.$$ In particular, both vector spaces in (\ref{eq:intedge}) naturally decompose into a direct sum of vector spaces indexed by the connected components of the internal edge graph. Therefore, in each connected component of the internal edge graph, the number of edges $e$ with $e \sqcup E$ interior is less than or equal to the number of vertices $v$ with $\{v\} \sqcup E$ of carrier codimension $1$. As the only connected graphs $(V, E)$ where $\vert E \vert \le \vert V \vert$ are either trees or contain a unique cycle, the result follows.
\end{proof}

\begin{lemma}\label{lem:nofourcycle}
Assume $\sigma(E)$ has codimension at least $2$. Let $F \subset \lk_{\Gamma}(E)$ be a face. Assume $F$ has no vertices $v$ with $\{v\} \sqcup E$ interior. If $L(\Gamma, E)|_{F}$ is zero in degree $2$, then no component of the internal edge graph of $F$ contains a cycle of length $4$. 
\end{lemma}
\begin{proof}
Suppose a component of the internal edge graph contains a $4$-cycle of vertices $F = \{t_1, t_2, u_1, u_2\}$. By Lemma~\ref{lem:intedge}, this is the unique cycle in this component and every vertex $w \in F$ has $\{w\} \sqcup E$ of carrier codimension $1$. Because $F$ is a face and there are no $3$-cycles in this component of the internal edge graph, we may assume that $\sigma(\{t_i\} \sqcup E) = \{v_1\}^c$ and $\sigma(\{u_i\} \sqcup E) = \{v_2\}^c$. Restricting to $F$ and using that $(L(\Gamma, E)|_{F})_2 = 0$, we have that 
$$(x^{\{ t_1, u_1 \}}, x^{ \{ u_1, t_2\}}, x^{\{t_2,u_2\}}, x^{\{u_2, t_1\}}) = (x^{\{t_1\}} \theta_{2}, x^{\{t_2\}} \theta_{2}, x^{\{u_1\}} \theta_{1}, x^{\{u_2\}} \theta_{1}).$$
The relation $\theta_{1} \theta_{2} - \theta_{2} \theta_{1} = 0$ expands into a relation between the generators of the right-hand side. But the left-hand side is $4$-dimensional, a contradiction.
\end{proof}

\begin{lemma}\label{lem:cd1case}
Assume $\sigma(E)$ has codimension $1$. Let $\Delta \subset \lk_{\Gamma}(E)$ be a subcomplex. Then 
$$\dim (L(\Gamma, E)|_{\Delta})_1 \ge |\{v \in \Delta: \{v\} \sqcup E \text{ interior}\}| - 1.$$
\end{lemma}
\begin{proof}
By considering the degree $1$ part of \eqref{e:restrict}, as the codimension of $\sigma(E)$ is $1$, 
we get the following exact sequence.
\begin{center}
\begin{tikzcd}[column sep = large]
 k \arrow[r] & \bigoplus\limits_{\mathclap{\substack{w \in \Delta \\ \{w\} \sqcup E \text{ interior }}}} k \cdot x^w \arrow[r] &   (L(\Gamma, E)|_{\Delta})_1  \arrow[r] & 0,
\end{tikzcd}
\end{center}
and the result follows. 
\end{proof}

\begin{proof}[Proof of Theorem~\ref{thm:localizedcase}]
We must show that $L(\Gamma, E)|_{F}$ is non-zero in degree $|F_1|$. 
Recall that $L(\Gamma, E)|_{F}$ is isomorphic to $I|_{F}/J|_{F}$, where $I|_{F}$ and $J|_{F}$ are described in \eqref{eq:Ires} and \eqref{eq:Jres} respectively. First we handle the cases when $\vert F_1 \vert \le 1$. If $F_1 = \emptyset$, then $E$ is interior and $x^{\emptyset} = 1$, but $J|_{F}$ is a proper ideal as it is generated by elements of positive degree, so $x^{F_1} \not \in J|_{F}$. If $F_1 = \{v\}$, then we assume that $E$ is not an interior face. Then $J|_{F}$ is generated by elements of degree at least $2$, so $x^{F_1} \not \in J|_{F}$.

Suppose $\vert F_1 \vert = 2$. We assume that there are no vertices $v$ with $\{v\} \sqcup E$ interior and $E$ is not interior. If $\sigma(E)$ has codimension $1$, then both $F_1$ and $F_2$ must have a vertex $v$ with $\{v\} \sqcup E$ interior. Then by Lemma~\ref{lem:cd1case}, we see that $\dim L(\Gamma, E)|_{F} \ge 1$. Hence we may assume that $\sigma(E)$ has codimension at least $2$.

Let $F_1 = \{u, t\}$ and assume that $L(\Gamma, E)|_{F}$ has no non-zero elements in degree $2$. Consider the connected component of the internal edge graph containing $F_1$. By Lemma~\ref{lem:intedge}, we may assume that $\sigma(\{u\} \sqcup E) = \{v_1\}^c$. Note that $v_1 \in \sigma(t)$. There is a vertex $t' \in F_2$ such that $v_1 \in \sigma(t')$, so $\{u, t'\} \sqcup E$ is interior. Therefore either $\{t\} \sqcup E$ or $\{t'\} \sqcup E$ has carrier codimension $1$. 

If $\sigma(\{t\} \sqcup E) = \{v_2\}^c$, then there is a vertex $u' \in F_2$ such that $v_2 \in \sigma(u')$. First assume $u'$ and $t'$ are distinct. Since at least one of $\{u'\} \sqcup E$ and $\{t'\} \sqcup E$ has carrier codimension $1$, it follows that $\{ u', t'\} \sqcup E$ is interior. Then $\{u, t, u', t'\}$ forms a $4$-cycle, contradicting Lemma~\ref{lem:nofourcycle}.

If $u' = t'$, then the internal edge graph contains a cycle and hence every vertex $w$ in it (including $t$) has $\{w\} \sqcup E$ of carrier codimension $1$. As $F_2$ is interior and $\{u'\} \sqcup E$ has carrier codimension $1$, there is a vertex $w \in F_2$ such that $\{u', w\} \sqcup E$ is interior. But then either $\{u, w\} \sqcup E$ or $\{t, w\} \sqcup E$ is interior, contradicting the uniqueness of the cycle in Lemma~\ref{lem:intedge}.

If $\{t\} \sqcup E$ does not have carrier codimension $1$, then we may assume that $\sigma(\{t'\} \sqcup E) = \{v_2\}^c$. Choose a vertex $u' \in F_2$ with $v_2 \in \sigma(u')$. Then $\{t', u'\} \sqcup E$ is interior, so $\{u'\} \sqcup E$ has carrier codimension $1$. If $v_1 \in \sigma(u')$, then $\{u, u'\} \sqcup E$ is interior. If $v_1 \not \in \sigma(u')$, then $\{t, u'\} \sqcup E$ is interior. In either case, there is a cycle and a vertex $v$ with $\{v\} \sqcup E$ of carrier codimension more than $1$ in the internal edge graph, contradicting Lemma~\ref{lem:intedge}. 
\end{proof}

\begin{remark}\label{rem:newnonvanish}
One can use the same overall strategy more generally to show that other combinatorial types of faces cannot appear in triangulations with vanishing local $h$-vectors.  For instance, suppose $V = \{v_1, \ldots, v_6\}$ and $\sigma \colon \Gamma \to 2^V$ is a geometric  triangulation with a facet $F = \{w_1, \ldots, w_6\}$ such that 
\[
\begin{array}{lll}
\sigma(w_1) = \{ v_1\} & \sigma(w_2) = \{ v_2 \} & \sigma(w_3) = \{ v_3 \} \\
\sigma(w_4) = \{ v_1, v_4, v_5 \} & \sigma(w_5) = \{ v_2, v_4, v_6 \} & \sigma(w_6) = \{ v_3, v_5, v_6 \}
\end{array}
\]
Then the interior $2$-faces of $F$ are $\{w_1, w_5, w_6\}$, $\{w_2, w_4, w_6\}$, $\{w_3, w_4, w_5\}$, and $\{w_4, w_5, w_6\}$. But $F$ has no interior vertices or edges, and it has only three edges with carrier codimension one, namely $\{w_4, w_5\}, \{w_4, w_6\},$ and $\{w_5, w_6\}$. Thus $L(\Gamma, \emptyset)|_{F}$ is non-zero in degree three. Note that $F$ is not a pyramid and does not admit an interior partition. 
\end{remark}

\bibliography{resolutions}
\bibliographystyle{amsalpha}

\end{document}